\documentclass[11pt]{amsart}
\usepackage{amstext}
\usepackage{amsthm}
\usepackage{amsopn}
\usepackage{amsfonts}
\usepackage{amsmath}
\usepackage{latexsym}
\usepackage{amscd} 
\usepackage{amssymb}
\usepackage{amsmath}
\swapnumbers

\newtheorem{thm}[equation]{Theorem}

\newtheorem{prop}[equation]{Proposition}

\newtheorem{lem}[equation]{Lemma}

\theoremstyle{definition}

\newtheorem{remark}[equation]{Remark}

\numberwithin{equation}{section}

\newcommand{\Nrd}{\operatorname{Nrd}}
\newcommand{\Trd}{\operatorname{Trd}}

\newcommand{\GG}{{\mathbb G}}

\newcommand{\ZZ}{{\mathbb Z}}
\newcommand{\PP}{{\mathbb P}}
\newcommand{\QQ}{{\mathbb Q}}
\newcommand{\RR}{{\mathbb R}}

\newcommand{\CC}{{\mathbb C}}

\newcommand{\Ker}{\operatorname{Ker}}
\newcommand{\Aut}{\operatorname{Aut}}
\newcommand{\Spec}{\operatorname{Spec}}

\newcommand{\GL}{{\operatorname{GL}}}
\newcommand{\SL}{{\operatorname{SL}}}
\newcommand{\SO}{{\operatorname{SO}}}

\newcommand{\Spin}{{\operatorname{Spin}}}

\newcommand{\simlgr}{\buildrel \sim \over \longrightarrow}

\newcommand{\PGL}{{\operatorname{PGL}}}

\newcommand{\calO}{\operatorname{\mathcal O}}
\newcommand{\calQ}{\operatorname{\mathcal Q}}

\newcommand{\tr}{\operatorname{tr}}

\newcommand{\uIsom}{\underline{\rm Isom}}

\newcommand{\uAut}{\underline{\rm Aut}}

\begin{document}

\title[Octonion algebras]{Octonion algebras over rings are not determined by their norms}

\author[P. Gille]{Philippe Gille}
\address{UMR 8552 du CNRS, DMA, Ecole
Normale Sup\'erieure, F-75005  Paris, France}
\email{Philippe.Gille@ens.fr} 
\date{\today}

\maketitle

\noindent{\bf R\'esum\'e:} R\'epondant \`a une question de H. Petersson, nous contruisons une classe
d'exemples de paires d'alg\`ebres d'octonions d\'efinies sur un anneau
ayant des normes isom\'etriques.
 
\medskip

\noindent{\bf Abstract:} Answering a question of H. Petersson, we provide
a class of examples  of pair of octonion algebras over a ring having isometric
norms.
 
\medskip


\medskip


\smallskip

\section{Introduction}

If $Q$ is  a quaternion  algebra over a field $k$, we know from Witt that
$Q$ is determined by its norm  \cite[\S 1.7]{SV}. This result  has been extended  over rings by Knus-Ojanguren-Sridharan (\cite[prop. 4.4]{KOS} ,  \cite[V.4.3.2]{K}) and holds actually over an arbitrary base (\S \ref{sect-quat}).

If $C$ is a octonion algebra over $k$, we know from van der Blij-Springer that 
it  is determined by its norm form \cite[claim 2.3]{BS} (see also \cite[\S 1.7]{SV}); more generally
it is true  over local rings (Bix, \cite[lemma 1.1]{B}). 
In his Lens lecture (May 21-25, 2012), H. Petersson raised the question 
whether it remains true over arbitrary commutative rings.

The goal of this note is to produce a  counterexample to this question, namely 
an example of two non-isomorphic octonions algebras over some commutative ring $R$
having isometric norms.  Our argument is based on the study of fibrations of group schemes
and uses topological fibrations which makes clear why it holds for quaternion algebras
and not for octonions.

\medskip

For the theory of reductive group schemes and related objects (e.g. Lie algebra sheaves,
homogeneous spaces, quadratic spaces, ... ) we refer to SGA3 \cite{SGA3}
and to the book by Demazure-Gabriel \cite{DG}. The sheaves in sets or groups are denoted as $\underline F$ and
are for the fppf (also called flat) topology over a base scheme $S$.

\medskip

\noindent{\bf Acknowledgments:}
We thank V. Chernousov and E. Neher for useful discussions.
We thank M. Brion for his nice remark (at the end).
Finally we thank the referee for his comments.

\section{Quaternion algebras and norms}\label{sect-quat}

Let $S$ be a scheme. 
By a quaternion\footnote{Knus' definition requests less conditions \cite[1.3.7]{K};
we deal here then with ``separable quaternions algebras''.} algebra over $S$, we mean 
a rank $4$ Azumaya $\calO_S$--algebra
$\calQ$.
Equivalently, it  is an \'etale $S$--form of the matrix algebra $M_2(\calO_S)$, namely the twist
 of $M_2(\calO_S)$ by the $\PGL_2$--torsor \break 
 $E=\uIsom_{alg}( M_2(\calO_S), \calQ )$.

By descent, it follows that isomorphism classes of quaternion $S$-algebras correspond
to the \'etale cohomology set $H^1(S, \PGL_2)$. 

The reduced norm (resp. trace) $\Nrd:  \calQ \to \calO_S$ (resp: $\Trd$) 
is the twist by $E$  of the determinant map
$M_2(\calO_S) \to \calO_S$ (resp. the trace),  it is a quadratic (resp. linear) form over $S$. 

Furthermore the  canonical involution $X \mapsto \tr(X)-X$ on $M_2(\calO_S)$ 
induces by descent the  canonical involution of $\calQ$. 

The $S$--group scheme $\SL_1(\calQ)$  (resp. $\PGL_2( \calQ)$, $\SO(\calQ, N_{\calQ})$)
is the twist by $E$ of  $\SL_2/S$  (resp. $\PGL_2/S$, $\SO(M_2, \det)/S$).

The point is that  the semisimple group scheme $\SO(\calQ, N_{\calQ})$ is of type $A_1 \times A_1$
and its universal cover is $\SL_1(\calQ) \times \SL_1(\calQ)$.

\begin{lem}
We have an exact sequence of group schemes
$$
1 \to \mu_2 \to \SL_1(\calQ) \times \SL_1(\calQ) \buildrel f \over \to \SO(\calQ, N_{\calQ}) \to 1  
$$
where $f(x,y).q= \, x q \, y^{-1}$ for every $q \in \calQ$.
\end{lem}

\begin{proof} We do first the case of $S=\Spec(\ZZ)$ and  $\calQ=M_2(\ZZ)$.
We have $\mu_2 \subset \ker(f)$ and let us show the converse inclusion. 
Let $R$ be a ring and pick $(x,y) \in \ker(f)(R)$. A such element satisfies $x A y^{-1}=A$ for each $A \in M_2(R)$.
By taking $A=y$, we see that $x=y$ so that  $x A x^{-1}=A$ for each $A \in M_2(R)$.
By taking the canonical $R$-basis of $M_2(R)$, it follows that $x \in \GG_m(R)$.
Since $x \in \SL_2(R)$, we conclude that $(x,y) \in \mu_2(R)$. Thus $\mu_2=\ker(f)$.

Since $\mu_2$ is a central subgroup of $\SL_2 \times_R \SL_2$, we can mod out by $\mu_2$ \cite[XXII.4.3]{SGA3}
and get a monomorphism $\widetilde f: (\SL_2 \times_R \SL_2)/ \mu_2 \to \SO(M_2,det)$ of semisimple
 group schemes. According to \cite[XVI.1.5.a]{SGA3}, it is a closed immersion.
On both sides,  each $\QQ$-fiber  is smooth connected of dimension 
$6$. It follows that  $\widetilde f_\QQ$ is an isomorphism. Since 
$\SO(M_2,det)$ is flat over $\ZZ$, we conclude that $\widetilde f$ is an isomorphism.  

The general case follows again by twisting everything by the $\PGL_2$--torsor $E$.

\end{proof}

The adjoint map $Ad: \PGL_2 \to \GL(M_2)$ 
gives rise to  the  closed $S$--immersion $\PGL_2 \to {\rm O}( M_2, \det)$
where ${\rm O}( M_2, \det)$ stands for the orthogonal group scheme of the nonsingular
quadratic form $\det$ \cite[III.5.2]{DG}.
It is equipped with the Dickson map ${\rm D}: {\rm O}( M_2, \det) \to \ZZ/2\ZZ$
whose kernel is  by definition the special linear group 
$\SO( M_2, \det)$.

By twisting by the torsor $E$, it provides a closed $S$-immersion 
$$
Ad: \PGL_1(\calQ) \to {\rm O}( \calQ, \Nrd), q \mapsto Ad(q)  
$$
where $\PGL_1(\calQ)$ stands for the group scheme $\GL_1(\calQ)/ \GG_m$
of projective units. 

On the other hand, the orthogonal $S$--group ${\rm O}( \calQ, \Nrd)$ acts
on $\SL_1(\calQ)=\Ker(\GL_1(\calQ)) \to \GG_m)$  by the action induced from the standard action of 
$\GL_1(\calQ)$ on $\calQ$.

\begin{prop}\label{Max} (1) The $S$--scheme $\SL_1(\calQ)$ is  a left homogeneous space 
(with respect to the flat topology)
under the action 
of  $\SO( \calQ, \Nrd)$ and a fortiori under the action of ${\rm O}( \calQ, \Nrd)$.

\smallskip

(2) The orbit map
$$
 u : \SO( \calQ, \Nrd) \to \SL_1(\calQ), \enskip g \mapsto g. 1 
$$
 is a split $\PGL_1(\calQ)$--torsor.

\end{prop}

\begin{proof} We put $G/S= \SO( \calQ, \Nrd)$, $H/S= \PGL_1(\calQ)$ and 
$X/S= \SL_1(\calQ)$.

\smallskip

\noindent (1) We have to check the definition \cite[IV.6.7]{SGA3}, namely to establish the properties

\medskip

(a) the map $G \times_S X \to X \times_S X$, $(g,x) \mapsto (x, g.x)$ 
is an epimorphism of flat sheaves;  

\smallskip

(b)  $f: X \to S$ has sections locally with respect to the flat topology.

\medskip

The condition (b) is obvious in our case since $f$ has a global section given by the unit of $X=\SL_1( \calQ)$.
Condition (a) will follow of  the following stronger condition:

\medskip

(c) $X(T)$ is homogeneous over $G(T)$ for each
$S$-scheme $T$.

\smallskip 

We are given $T/S$ and a couple of quaternions $q_1,q_2 \in X(T)$ of reduced norm one.
We put $q=q_2 \, q_1^{-1} \in X(T)$.
The left translation  $L_q$ is an element of $G(T)$ which satisfies
$L_q.q_1=q_2$. This shows (c).

\smallskip

\noindent (2)    
The map $u \circ f:  \SL_1(\calQ) \times \SL_1(\calQ) \to \SL_1(\calQ)$ reads as follows: 
$(u \circ f) (x,y)=xy^{-1}$. Therefore $\SL_1(\calQ) \times_S \SL_1(\calQ) / \SL_1(\calQ) \simlgr \SL_1(\calQ)$
where    $\SL_1(\calQ)$ acts on $\SL_1(\calQ) \times_S \SL_1(\calQ)$ by $z. (x,y)= (x\,z, z^{-1}\,x)$.
By modding out by the diagonal $\mu_2$ of $\SL_1(\calQ) \times_S \SL_1(\calQ)$,
 we get an isomorphism of flat sheaves $$
\SO( \calQ, \Nrd)/ \PGL_1(\calQ) \simlgr \SL_1(\calQ)
$$ 
where $\PGL_1(\calQ)$ embeds by $h$ in $\SO( \calQ, \Nrd)$.
\end{proof}

\smallskip

\begin{lem} ${\rm O}(\calQ, \Nrd)= \SO( \calQ, \Nrd) \times_S \ZZ/2\ZZ$ where
$\ZZ/2\ZZ$ is the $S$-subgroup ${\rm O}(\Nrd)$ defined by the canonical involution.
\end{lem}

\begin{proof} We have to show  that the Dickson map
${\rm D} :{\rm O}( \calQ, \Nrd) \to \ZZ/2\ZZ$ is split by applying $1$ to the canonical 
involution. To check that the Dickson invariant of the canonical 
involution is $1$, we can reason \'etale locally, that is to check it for 
each strict henselization $O_{S,s}^{sh}$ where $s$ is a point of $S$.  
In particular, it enables us to assume that $\calQ$ is the split 
quaternion algebra which is defined over $\ZZ$.

We can then deal with $S=\Spec(\ZZ)$ and $\calQ=M_2(\ZZ)$ and it remains to show
that ${\rm D}(\sigma)= 1$ where $\sigma$ is the canonical involution of $M_2(\ZZ)$.
It is enough to check it over $\QQ$ and then the Dickson invariant is nothing but
the determinant by means of the identification $(\ZZ/2\ZZ)_\QQ \cong \mu_{2, \QQ}$ \cite[III.5.2.6]{DG}. 
The  basis 
$$
\left[\begin{array}{cc}
1& 0 \\
0&1 \\
\end{array}
\right], \enskip
\left[\begin{array}{cc}
-1& 0 \\
0&-1 \\
\end{array}
\right], \enskip
\left[\begin{array}{cc}
0& 1 \\
1& 0 \\
\end{array}
\right], \enskip
\left[\begin{array}{cc}
0& - 1 \\
1& 0 \\
\end{array}
\right]
$$
of $M_2(\QQ)$  is a diagonalization   basis for $\sigma$ whose  eigenvalues are  $1,-1,-1,-1$.
The determinant of $\sigma$ is then $-1$, as desired.

\end{proof}

If follows that we have an isomorphism of homogeneous ${\rm O}( \calQ, \Nrd)$-spaces
${\rm O}( \calQ, \Nrd) /  ( \PGL_1(\calQ) \times_S \ZZ/2\ZZ)  \, \simlgr \, \SL_1(\calQ)$.

\begin{thm}\label{isometric}  Let $\calQ'$ be a $\calO_S$-quaternion algebra.
Then $\calQ'$ is isomorphic to $\calQ$ if and only if
the quadratic $S$--form $\Nrd$ and $\Nrd'$ are isometric.
\end{thm}

\begin{proof} Since $H^1(S, \PGL_1(\calQ))$ classifies $S$--quaternion algebras and 
$H^1(S, {\rm O}(\calQ, \Nrd))$ classifies the isometry classes of  nonsingular quadratic forms of dimension $4$, 
it follows that the kernel of the map
$$
Ad_*: H^1(S, \PGL_1(\calQ)) \to  H^1(S,  {\rm O}(\calQ, \Nrd))
$$
classifies the  isomorphism classes of quaternions $S$--algebras  
such that  the quadratic $S$--form $\Nrd$ and $\Nrd'$ are isometric.
By applying \cite[III.3.2.2]{Gir} to
the isomorphism ${\rm O}( \calQ, \Nrd) /  ( \PGL_1(\calQ) \times_S \ZZ/2\ZZ)  \, \simlgr \, \SL_1(\calQ)$, 
 we get an exact sequence of pointed sets
$$
{\rm O}( \calQ, \Nrd)(S) \buildrel f \over \to 
\SL_1(\calQ)(S) \to 
H^1(S, \PGL_1(\calQ) \times_S \ZZ/2\ZZ ) \to  H^1(S,  {\rm O}(\calQ, \Nrd)).
$$
By proposition \ref{Max}, the map $f$ admits a retraction so that 
the kernel of \break
 $H^1(S, \PGL_1(\calQ) \times_S \ZZ/2\ZZ ) \to  H^1(S,  {\rm O}(\calQ, \Nrd))$
is trivial. A fortiori, the kernel of $H^1(S, \PGL_1(\calQ)  ) \to  H^1(S,  {\rm O}(\calQ, \Nrd))$
is trivial, as desired.
\end{proof}

\begin{remark} {\rm Knus-Ojanguren-Sridharan's proof uses the even Clifford algebra of the norm forms
to encode the algebra. Somehow we use also the Clifford algebra by means 
of the Dickson invariant which is in the case related to the fact that the simply connected cover of  
$\SO(\calQ,N_{\calQ})$ is  $\SL_1(\calQ) \times_S \SL_1(\calQ)$. 
}
\end{remark}

\section{Octonion  algebras and norms}

Let $R$ be a commutative ring (with unit). From \S 4 of \cite{LPR}, 
a non-associative algebra $C$ over $R$ is called 
an octonion $R$-algebra\footnote{One can of course
globalize this definition, see \cite{P}.} if  it is 
 a finitely generated projective $R$--module of rank $8$,
 contains an identity element
$1_C$ and admits a norm, i.e. a map  $n_C: C \to R$
satisfying  the two following  conditions: 

\medskip

(1) $n_C$ is a nonsingular quadratic form;

\medskip

(2)   $n_C(xy)=n_C(x) \, n_C(y)$ for all $x,y \in C$. 

\medskip

This notion is stable under base  extension and  descends  under
faithfully flat base change of rings.

\medskip

The basic example of an octonion algebra is the split octonion algebra ({\it ibid}, 4.2)
denoted  $C_0$ and   called the algebra of Zorn vector matrices, which is defined over $\ZZ$.
There is another description of this algebra  in \S 1.8 of \cite{SV} over fields
 by the ``doubling process''. It actually  works over $\ZZ$,  we  take 
$$
C'_0= M_2(\ZZ) \oplus M_2(\ZZ)
$$
with multiplication law $(x,y).(u.v)= ( x\, u + v   \, \sigma(y), \sigma(x) v + u \, y )$ ($\sigma$ is
 the canonical involution of $M_2(\ZZ)$) and norm 
$n_{C'_0}(x,y)=\det(x) - \det(y)$.
We know that the fppf $\ZZ$--group sheaf $\uAut(C_0)\cong \uAut(C'_0)$ is representable by 
an affine smooth group $\ZZ$--scheme  $\Aut(C_0)$ \cite[4.10]{LPR}.

\begin{prop} The $\ZZ$--group scheme $\Aut(C_0)$ is the Chevalley group of type $G_2$.
\end{prop}

\begin{proof} Let us first show  that $\Aut(C_0)$ is a semisimple group scheme of type $G_2$, that 
is by definition a smooth affine group scheme whose geometrical fibers 
 are semisimple  groups of type $G_2$ \cite[XIX]{SGA3}. 

The fibers of the affine smooth group $\ZZ$--scheme  $\Aut(C_0)$ are 
indeed semisimple  groups of type $G_2$ according to  theorem 2.3.5 of \cite{SV}.
Hence  $\Aut(C_0)$ is a semisimple group scheme of type $G_2$.
By Demazure's unicity theorem \cite[cor. 5.5]{SGA3}
the Chevalley group of type $G_2$ is  the unique split semisimple group scheme of type $G_2$, that is 
the unique semisimple group scheme of type $G_2$ admitting a split torus of rank two. 
Since ${\PGL_2 \times \PGL_2}$ embeds in $\Aut(C'_0)$,  $\Aut(C'_0)$ contains a two dimensional split 
  torus. Thus  $\Aut(C_0)\cong \Aut(C'_0)$ is the Chevalley group of type $G_2$. 
\end{proof}

We come now to the question whether an octonion algebra is determined
by its norm. 
Let $C$ be an octonion algebra over $R$. We have  natural closed group embeddings of group schemes
$$
\Aut(C) \buildrel j \over \to   {\rm O}(n_C) \subset \GL(C).
$$ 
We get a map in cohomology 
$$
j_*: H^1(R, \Aut(C)) \to H^1(R,{\rm O}(n_C)).
$$
The left handside classifies octonion algebras over $R$
while the right handside classifies $8$-dimensional nonsingular quadratic $R$--forms.
By descent, we have $j_*([C'])= [n_{C'}]$ for each octonion $R$--algebra $C'$.
It follows that the kernel of $j_*$ classifies the octonion algebras over $R$ whose 
norm form is isometric to $n_C$.

\begin{lem} The fppf quotient ${\rm O}(n_C)/ \Aut(C)$ is representable by an affine scheme
of finite presentation over $R$. 
\end{lem}

\begin{proof}  
According to \cite[6.12]{CTS}, the fppf quotient $\GL(C)/ \Aut(C)$ is representable by an affine scheme of finite type over $R$. 
It is of finite presentation over $R$ by the standard limit argument \cite[VI$_B$.10.2]{SGA3}.
 In the other hand, the fppf sheaf
 $\GL(C) / {\rm O}(n_C)$ is representable by an affine scheme of finite presentation over $R$
\cite[lemme 2.26]{SZ}. Therefore the ``kernel''  ${\rm O}(n_C)/ \Aut(C)$
of the natural map  $ \GL(C)/ \Aut(C) \to \GL(C) / {\rm O}(n_C)$ is representable
by an affine scheme of finite type.
\end{proof}

We denote by $A(C)$ the coordinate ring of the affine scheme ${\rm O}(n_C)/ \Aut(C)$.

\begin{thm}\label{main} Assume that $R$ is a non trivial  $\QQ$-ring. Then 
the $\Aut(C)$-torsor ${\rm O}(n_C) \to \Spec(A(C))$ is not trivial, so that 
$\ker(j_{*,A(C)})$ is not trivial.
\end{thm}

\begin{remark}{\rm By inspection of the proof, the result holds also for 
 ${\rm SO}(n_C) \to  {\rm SO}(n_C)/ \Aut(C)$. If $R= \CC$, it 
provides then a counterexample over a connected smooth complex affine variety.   
 }
\end{remark}

Let us do first a special  case.

\begin{prop}\label{homotopy} Let $C/\RR$ be the ``compact'' Cayley octonion algebra. 
Then theorem \ref{main} holds is this case.
\end{prop}

\begin{proof} In this case $G=\Aut(C)/\RR$ is the  anisotropic real form of $G_2$ and
we consider its embedding in the ``compact''  ${\rm O}_8$.
We reason by  contradiction assuming that 
 the $G$--torsor   ${\rm O_8} \to {\rm O_8}/G$ is split. It follows that there is a $G$-equivariant isomorphism
${\rm O_8}  \cong {\rm O_8}/G \times_\RR G$ over ${\rm O_8}/G$. Hence the map $G \to {\rm O_8}$ admits 
a section.  
Taking the real points,  it follows that  the map $G(\RR) \to {\rm O_8}(\RR)$ admits a continuous section, hence 
the homotopy group  $\pi_n( G(\RR), \bullet)$ is a direct summand of   $\pi_n( {\rm O_8}(\RR), \bullet)$ for all $n \geq 1$. 
From the tables \cite[p. 970]{M}, we have
 $\pi_6( G(\RR),\bullet) \cong \ZZ/3\ZZ$ 
and $\pi_n( {\rm O_8}(\RR), \bullet) = \pi_n( {\rm SO_8}(\RR),\bullet) =0$, hence a contradiction.

\end{proof}

We can proceed to the proof of theorem \ref{main}.

\begin{proof}
We claim that  the above counterexample survives when extending the scalars
to  $\CC$. According to the Cartan decomposition, there are 
homomeorphisms $G(\CC) \cong G(\RR) \times \RR^m$ and
 ${\rm O}_8(\CC) \cong  {\rm O}_8(\RR) \times \RR^n$
Hence $\pi_6(G(\CC), \bullet)= \ZZ/3\ZZ$ and does not  inject in $\pi_6({\rm O}_8(\CC),\bullet)=0$.

In other words, theorem \ref{main} holds for the case $R=\Spec(\CC)$ and $C=C_0$.
It holds over  $\overline \QQ$ and over an arbitrary
algebraically closed field of characteristic zero.

For the general case, we consider a morphism $R \to F$ where $F$ is an algebraically closed field.
Since  the $\Aut(C)_F$-torsor ${\rm O}(n_C)_F \to {\rm O}(n_C)_F/ \Aut(C)_F$ is not split, it follows
that   the   $\Aut(C)_F$-torsor ${\rm O}(n_C) \to {\rm O}(n_C)/ \Aut(C)$ is not split.
\end{proof}

\medskip

\noindent{\bf Concluding remarks.}
(1) The rings occuring in the examples are  of dimension 14. The next question is to know
the minimal dimension for the counterexamples. M. Brion communicates us a smaller example, say over the
 complex numbers. Since the action of map $G_2$ on the complex octonions $C$ preserves 
$1_C$ and the octonions of trace $0$, the map $G_2 \to  \SO_8$  takes value in $\SO_7 \subset \SO_8$.
A fortiori the $G_2$--torsor $\SO_7 \to \SO_7/G_2=\Spec(B)$ provides an example of 
a non trivial octonion algebra over $B$ having trivial norm. The dimension of $B$ is 
then $7$. Also the homogeneous space  $\SO_7/G_2$  occurs as the complement 
of a smooth quadric in  $\PP^7$. Let us  explain this geometric fact.
Firstly  the map $G_2 \to \SO_7$ lifts in $G_2 \to \Spin_7$.
The spinorial action of $\Spin_7$  on $\CC^7$  has been investigated by Igusa
 \cite[prop. 4]{I}. The $\Spin_7$--orbits in $\CC^7$ are $0$, 
the orbit of a vector of highest weight and a one parameter family 
of closed orbits with stabilizers $G_2$,
defined  by an equation  $g(x) = t$ where $g$ is an invariant
quadratic form. It follows that the induced action of $\SO_7$ on the projective space 
$\PP^7$ has two orbits, one open $\SO_7/G_2$ and one closed which is a smooth projective
quadric.

\smallskip

(2) For the ring $\ZZ$, van der Blij-Springer showed that  there are only  
two octonions algebras  and  having distinct norm forms \cite[\S 4]{BS} (see also \cite{CG}).
Hence octonion algebras over $\ZZ$ are determined by their 
norms. 
For other rings of integers, it seems to be an open question.


\medskip

\begin{thebibliography}{EGA4}



\bibitem{BS} F.  van der Blij, T. A. Springer, {\it 
The arithmetics of octaves and of the group $G_2$}
Nederl. Akad. Wetensch. Proc. Ser. A 62 = Indag. Math. {\bf 21} (1959), 406--418. 


\bibitem{B} R. Bix, {\it Isomorphism theorems for octonion planes over local rings},
Trans. Amer. Math. Soc. {\bf 266} (1981), 423--439.


\bibitem{CG} B. Conrad, B. Gross, {\it Non-split reductive groups over $\ZZ$},
preprint (2011).

\bibitem{CTS} J.--L. Colliot--Th\'el\`ene, J.--J. Sansuc, {\it Fibr\'es
quadratiques et composantes connexes r\'eelles}, Math. Annalen {\bf
244} (1979), 105--134.


\bibitem{DG} M.  Demazure, P. Gabriel,
{\em Groupes alg\'ebriques}, North-Holland (1970).




\bibitem{GS} P. Gille and T. Szamuely, {\it Central simple algebras 
and Galois cohomology}, Cambridge Studies in Advanced
Mathematics {\bf 101} (2006), Cambridge University Press.



\bibitem{Gir} J.  Giraud,
{\em Cohomologie non-ab\'elienne}, Springer (1970).

\bibitem{I} J.-I. Igusa, {\it 
A classification of spinors up to dimension twelve}, 
Amer. J. Math. {\bf 92}  (1970),  997--1028. 


\bibitem{K} M.-A. Knus, {\it Quadratic and hermitian forms over
rings}, Grundlehren der mathematischen Wissenschaften {\bf 294} (1991),
Springer.


\bibitem{KOS} M.-A. Knus, M. Ojanguren, R. Sridharan, {\it Quadratic forms and Azumaya algebras},  J. reine angew. math. {\bf 303/304} (1978),  231--248.

\bibitem{LPR} O. Loos, H.P Petersson, M.L. Racine, {\it Inner derivations
of alternative algebras over commutative rings}, Algebra and Number Theory 
{\bf 2} (2008), 927--968.

\bibitem{M} M. Mimura, {\it Homotopy theory of Lie groups}, Handbook of algebraic topology, North-Holland, Amsterdam (1995), 951--991.



\bibitem{P} H. Petersson, {\it  
Composition algebras over algebraic curves of genus zero}, 
Trans. Amer. Math. Soc. {\bf 337} (1993),  473--493. 


\bibitem{SGA3} {\it S\'eminaire de G\'eom\'etrie alg\'ebrique de
l'I.H.E.S., 1963-1964, sch\'emas en groupes, dirig\'e par M.
Demazure et A. Grothendieck},  Lecture Notes in Math. 151-153.
Springer (1970).

\bibitem{SV} T.A. Springer, F.D. Veldkamp, {\it Octonions algebras, Jordan algebras and 
exceptional groups}, Springer Monographs in Mathematics (2000 ).


\bibitem{SZ} A. Steinmetz-Zikesch, {\it  Alg\`ebres de Lie de dimension infinie et
 th\'eorie de la descente}, \`a para\^ {\i}tre aux M\'emoires de la Soci\'et\'e 
Math\'ematique de France. 

\end{thebibliography}
 \end{document}